\documentclass[12pt]{amsart}
\usepackage{amssymb,latexsym,amsmath}
\begin{document}

\begin{abstract}
This is a semi-expository article concerning 
Langlands functoriality and Deligne's conjecture on the special values of 
$L$-functions. The emphasis is on symmetric power $L$-functions associated
to a holomorphic cusp form, while appealing to a recent work of Mahnkopf
on the special values of automorphic $L$-functions.
\end{abstract}

\title[\bf Special values of $L$-functions]
{\bf Functoriality and special values of $L$-functions}

\author{\bf A. Raghuram \ \ \and \ \  \bf Freydoon Shahidi}

\date{\today}

\maketitle

\tableofcontents

\numberwithin{equation}{section}
\newtheorem{thm}[equation]{Theorem}
\newtheorem{cor}[equation]{Corollary}
\newtheorem{lemma}[equation]{Lemma}
\newtheorem{prop}[equation]{Proposition}
\newtheorem{con}[equation]{Conjecture}
\newtheorem{ass}[equation]{Assumption}
\newtheorem{defn}[equation]{Definition}
\newtheorem{rem}[equation]{Remark}
\newtheorem{exer}[equation]{Exercise}
\newtheorem{exam}[equation]{Example}
\newtheorem{quest}[equation]{Question}

\section{Introduction}

Langlands functoriality principle reduces the study of automorphic forms
on the ad\`elic points of a reductive algebraic group to those of an 
appropriate general linear group. In particular, every automorphic 
$L$-function on an arbitrary reductive group must be one for a 
suitable ${\rm GL}_n$. One should therefore be able to reduce the study of
of special values of an automorphic $L$-function to those of
a principal $L$-function of Godement and Jacquet on ${\rm GL}_n$. 

While the integral representations of Godement and Jacquet do not seem  
to admit a cohomological interpretation, there is a recent work of 
J. Mahnkopf \cite{mahnkopf1} \cite{mahnkopf2} which provides us with 
such an interpretation for certain Rankin--Selberg type integrals. 
In particular, modulo a nonvanishing assumption on local archimedean 
Rankin--Selberg product $L$-functions for forms on 
${\rm GL}_n \times {\rm GL}_{n-1}$, he defines a pair of periods, which
seem to be in accordance with those of Deligne \cite{deligne} 
and Shimura \cite{shimura3}. This work of Mahnkopf is quite remarkable
and requires the use of both Rankin--Selberg and Langlands--Shahidi methods
in studying the analytic (and arithmetic) properties of $L$-functions. His work
therefore brings in the theory of Eisenstein series to play an important role. 
In \S\ref{sec:mahnkopf} we briefly review this work of Mahnkopf.

This article is an attempt to test the philosophy--to study the special 
values of $L$-functions while using functoriality--by means of recent 
cases of functoriality established for symmetric powers of automorphic forms
on ${\rm GL}_2$ \cite{kim-jams} \cite{kim-shahidi-annals2}. 
While a proof of the precise formulae in the conjectures of Deligne
\cite{deligne} still seem to be out of reach, we expect to be able to 
prove explicit connections between the special values of symmetric power
$L$-functions twisted by Dirichlet characters and those of the original 
symmetric power $L$-functions using this work of Mahnkopf. These 
relations are formulated in this paper as Conjecture~\ref{con:twisted}
which also follows from the more general conjectures of 
Blasius \cite{blasius2} and Panchiskin \cite{panchiskin}, although the
heuristics underlying our conjecture are quite different.

A standard assumption made in the study of special values of $L$-functions 
is that the representations (to which are attached the $L$-functions) are 
cohomological. This is the case in Mahnkopf's work. A global 
representation being cohomological is entirely determined by the 
archimedean components. For representations which are symmetric power
lifts of a cusp form on ${\rm GL}_2$ we have the following fact. 
Consider a holomorphic cusp form on the upper half plane of weight $k$. 
This corresponds to a cuspidal automorphic representation, which is 
cohomological if $k \geq 2$, and any symmetric power lift, if cuspidal, is 
essentially cohomological. 
See Theorem~\ref{thm:symmetric-cohomology}.
(If the weight $k =1$ then the representation is not 
cohomological, and furthermore none of the symmetric power $L$-functions 
have any critical points.) In \S\ref{sec:cohomology}
we review representations with cohomology in the case of ${\rm GL}_n$.

We recall the formalism of Langlands functoriality for symmetric powers in 
\S\ref{sec:functorial}. We then review Deligne's conjecture for the 
special values of symmetric power $L$-functions in \S\ref{sec:deligne}
and give a brief survey as to which cases are known so far. In 
\S\ref{sec:dihedral} we sketch a proof of the conjecture for 
dihedral representations. We add that in the dihedral case the validity
of the conjecture in known to experts.

\smallskip

{\small {\it Acknowledgments:} This paper is based on a talk given by the 
second author during the workshop on Eisenstein series and Applications at 
the American Institute of Mathematics (AIMS) in August of 2005; he would 
like to thank the organizers Wee Teck Gan, Stephen~Kudla and Yuri 
Tschinkel as well as Brian Conrey of AIMS for a most productive meeting. 
Both the authors thank Laurent Clozel, Paul Garrett, Joachim Mahnkopf and 
Dinakar~Ramakrishnan for helpful discussions and email correspondence. 
We also thank the referee for a very careful reading and for making
suggestions to improve the content of the paper. 
The work of the second author is partially supported by NSF grant 
DMS-0200325.}

\section{Symmetric powers and functoriality}
\label{sec:functorial}

In this section we recall the formalism of Langlands functoriality
especially for symmetric powers. We will be brief here as there are 
several very good expositions of the principle of functoriality; 
see for instance \cite[Chapter 2]{clozel2}.  

Let $F$ be a number field and let 
${\mathbb A}_F$ be its ad\`ele ring. We let $\pi$ be a cuspidal 
automorphic representation of ${\rm GL}_2({\mathbb A}_F)$, by which we mean
that, for some $s \in {\mathbb R}$, $\pi \otimes |\cdot |^s$ is an 
irreducible summand of
$$
L^2_{\rm cusp}({\rm GL}_2(F)\backslash
{\rm GL}_2({\mathbb A}_F), \omega)
$$
the space of square-integrable cusp forms with central character $\omega$. 
We have the decomposition
$\pi = \otimes_v' \pi_v$ where $v$ runs over all places of $F$ and 
$\pi_v$ is an irreducible admissible representation of ${\rm GL}_2(F_v)$.

The local Langlands correspondence for ${\rm GL}_2$ (see \cite{kutzko} 
and \cite{kudla} for the $p$-adic case and \cite{knapp} for the 
archimedean case), 
says that to $\pi_v$ is associated a representation 
$\sigma(\pi_v) : W_{F_v}' \to {\rm GL}_2({\mathbb C})$ of the 
Weil--Deligne group $W_{F_v}'$ of $F_v$. (If $v$ is infinite, we take 
$W_{F_v}' = W_{F_v}$.) Let $n \geq 1$ be an integer. Consider the 
$n$-th symmetric power of $\sigma(\pi_v)$ which is an $n+1$ dimensional
representation. This is simply the composition of $\sigma(\pi_v)$ with 
${\rm Sym}^n : {\rm GL}_2({\mathbb C}) \to {\rm GL}_{n+1}({\mathbb C})$. 
Appealing to the local Langlands correspondence for ${\rm GL}_{n+1}$
(\cite{harris-taylor}, \cite{henniart}, \cite{knapp}, \cite{kudla})
we get an irreducible admissible representation of 
${\rm GL}_{n+1}(F_v)$ which we denote as ${\rm Sym}^n(\pi_v)$.  
Now define a global representation of ${\rm Sym}^n(\pi)$ of 
${\rm GL}_{n+1}({\mathbb A}_F)$ by 
$$
{\rm Sym}^n(\pi) := \otimes'_v \ {\rm Sym}^n(\pi_v).
$$

{\it Langlands principle of functoriality} predicts that 
${\rm Sym}^n(\pi)$ is an automorphic representation of ${\rm 
GL}_{n+1}({\mathbb A}_F)$, i.e., it is isomorphic to an irreducible 
subquotient of the representation of ${\rm GL}_{n+1}({\mathbb A}_F)$ 
on the space of automorphic forms \cite[\S4.6]{borel-jacquet}. 
If $\omega_{\pi}$ is the central character of $\pi$ then 
$\omega_{\pi}^{n(n+1)}$ is the central character of ${\rm Sym}^n(\pi)$. 
Actually it is expected to be an isobaric automorphic representation. 
(See \cite[Definition 1.1.2]{clozel2}
for a definition of an isobaric representation.) 
The principle of functoriality for the $n$-th symmetric power 
is known for $n=2$ by Gelbart--Jacquet \cite{gelbart-jacquet};
for $n=3$ by Kim--Shahidi \cite{kim-shahidi-annals2}; and 
for $n=4$ by Kim \cite{kim-jams}. For certain special forms $\pi$, for 
instance, if $\pi$ is dihedral, tetrahedral, octahedral or icosahedral, it 
is know for all $n$ (see \cite{kim-inventiones} \cite{raghuram-shahidi}).

\section{Deligne's conjecture for symmetric power $L$-functions}
\label{sec:deligne}

Deligne's conjecture on the special values of $L$-functions is a  
conjecture concerning the arithmetic nature of 
special values of motivic $L$-functions at critical points. 
The definitive reference is 
Deligne's article \cite{deligne}. We begin by introducing the symmetric 
power $L$-functions, which are examples of motivic $L$-functions, and then 
state Deligne's conjecture for these $L$-functions.

\subsection{Symmetric power $L$-functions}
Let $\varphi \in S_k(N,\omega)$, i.e., $\varphi$ is a
holomorphic cusp form on the upper half plane,
for $\Gamma_0(N)$, of weight $k$, and nebentypus character
$\omega$. Let $\varphi(z) = \sum_{n=1}^{\infty} a_nq^n$ be the Fourier
expansion of $\varphi$ at infinity.
We let $L(s, \varphi)$ stand for the completed $L$-function associated to
$\varphi$ and let $L_f(s, \varphi)$ stand for its finite part.
Assume that $\varphi$ is a primitive form in $S_k(N,\omega)$.
By primitive, we mean that it is an eigenform, a newform and is normalized
such that $a_1(\varphi)=1$. In a suitable right half plane 
the finite part $L_f(s, \varphi)$ is a Dirichlet series with an Euler 
product 
$$
L_f(s, \varphi) =
\sum_{n=1}^{\infty} a_n n^{-s} =
\prod_p L_p(s, \varphi),
$$
where, for all primes $p$, we have
$$
L_p(s, \varphi) =
(1-a_pp^{-s} + \omega(p)p^{k-1-2s})^{-1} =
(1-\alpha_{p,\varphi} p^{-s})^{-1} (1-\beta_{p,\varphi} p^{-s})^{-1},
$$
with the convention that if $p|N$ then $\beta_{p,\varphi} = 0$. 
Let $\eta$ be a Dirichlet character modulo an integer $M$. 
We let $S_f$ stand for the set of primes dividing $NM$ and let
$S = S_f \cup \{\infty\}$.
For any $n \geq 1$, the twisted partial $n$-th symmetric power $L$-function 
of $\varphi$ is defined as
$$
L^S(s, {\rm Sym}^n \varphi, \eta) =
\prod_{p \notin S} L_p(s, {\rm Sym}^n \varphi, \eta ),
$$
where, for all $p \notin S$, we have
$$
L_p(s,{\rm Sym}^n \varphi, \eta) =
\prod_{i=0}^n (1-\eta(p)\alpha_{p,\varphi}^i \beta_{p,\varphi}^{n-i}
p^{-s})^{-1}.
$$
If $M=1$ and $\eta$ is (necessarily) trivial then we denote the corresponding $L$-function as 
$L^S(s, {\rm Sym}^n \varphi)$.

Using the local Langlands correspondence the partial $L$-function 
can be completed by defining local factors $L_p(s, {\rm Sym}^n \varphi, \eta)$
for $p \in S$ and the completed $L$-function, which is a product over
all $p$ including $\infty$, will be denoted as $L(s, {\rm Sym}^n \varphi, \eta)$.  
The Langlands program predicts that $L(s, {\rm Sym}^n \varphi, \eta)$, which
is initially defined only in a half plane, admits a meromorphic
continuation to the entire complex plane and that it has all the usual
properties an automorphic $L$-function is supposed to have. This is known
for $n \leq 4$ from the works of several people including Hecke, Shimura,
Gelbart--Jacquet, Kim and Shahidi. It is also known for all $n$
for cusp forms of a special
type, for instance, if the representation corresponding to the cusp form
is dihedral or the other polyhedral types. (The reader is referred to the
same references as in the last paragraph of the previous section.)

\subsection{Deligne's conjecture}
Let $\varphi$ be a primitive form in $S_k(N,\omega)$.
Let $M(\varphi)$ be the motive associated to $\varphi$. This is a rank
two motive over ${\mathbb Q}$ with coefficients in the field
${\mathbb Q}(\varphi)$ generated by the Fourier coefficients of $\varphi$.
(We refer the reader to Deligne \cite{deligne} and 
Scholl \cite{scholl} for details about $M(\varphi)$.)
The $L$-function $L(s, M(\varphi))$ associated to this motive is
$L(s, \varphi)$.

Given the motive $M(\varphi)$, and given any complex embedding
$\sigma$ of ${\mathbb Q}(\varphi)$, there are nonzero complex numbers,
called Deligne's periods, $c^{\pm}(M(\varphi),\sigma)$ associated to it.
For simplicity of exposition we concentrate on the case when $\sigma$
is the identity map, and denote the corresponding periods as 
$c^{\pm}(M(\varphi))$. We note 
that this ``one component version" of Deligne's conjecture 
that we state below is not as strong as the full conjecture and so we are dealing
with a simplistic situation. 

Similarly, for the symmetric powers
${\rm Sym}^n(M(\varphi))$, we have the corresponding periods
$c^{\pm}({\rm Sym}^n(M(\varphi))$. In \cite[Proposition 7.7]{deligne}
the periods for the symmetric powers are related to the
periods of $M(\varphi)$. The explicit formulae therein have
a quantity
$\delta(M(\varphi))$ which is essentially the Gauss sum of the nebentypus
character $\omega$ and is given by
$$
\delta(M(\varphi)) \sim
(2 \pi i)^{1-k}\mathfrak{g}(\omega) :=
(2 \pi i)^{1-k}
\sum_{u=0}^{c-1}
\omega_0 (u) {\rm exp}(-2\pi i u/c),
$$
where $c$ is the conductor of $\omega$, 
$\omega_0$ is the primitive character associated to $\omega$, and 
by $\sim$ we mean up to an element of ${\mathbb Q}(\varphi)$. 
We will denote the right hand side by $\delta(\omega)$. For brevity, we 
will
denote $c^{\pm}(M(\varphi))$ by $c^{\pm}(\varphi)$.
Recall \cite[Definition 1.3]{deligne} that an integer $m$ is 
{\it critical} for
any motivic $L$-function $L(s,M)$ if
both $L_{\infty}(s,M)$ and $L_{\infty}(1-s, M^{\vee})$ are regular at 
$s=m$. We now state Deligne's conjecture \cite[Section 7]{deligne}
on the special values of the symmetric power
$L$-functions.

\begin{con}
\label{conj:deligne}
Let $\varphi$ be a primitive form in $S_k(N,\omega)$.
There exist nonzero complex numbers $c^{\pm}(\varphi)$ such that
\begin{enumerate}
\item If $m$ is a critical integer for $L_f(s, {\rm Sym}^{2l+1}\varphi)$, 
then
$$
L_f(m, {\rm Sym}^{2l+1}\varphi) \sim
(2\pi i)^{m(l+1)}\, c^{\pm}(\varphi)^{(l+1)(l+2)/2}\,
c^{\mp}(\varphi)^{l(l+1)/2}\, \delta(\omega)^{l(l+1)/2},
$$
where $\pm = (-1)^m$.

\item If $m$ is a critical integer for $L_f(s, {\rm Sym}^{2l}\varphi)$,
then
$$
L_f(m, {\rm Sym}^{2l}\varphi) \sim
\left\{\begin{array}{ll}
(2\pi i)^{m(l+1)}\,
(c^+(\varphi)c^-(\varphi))^{l(l+1)/2}\,
\delta(\omega)^{l(l+1)/2}
& \mbox{if $m$ is even,} \\
(2\pi i)^{ml}\,
(c^+(\varphi)c^-(\varphi))^{l(l+1)/2}\,
\delta(\omega)^{l(l-1)/2}
& \mbox{if $m$ is odd.}
\end{array}\right.
$$
\end{enumerate}
By $\sim$ we mean up to an element of ${\mathbb Q}(\varphi)$.
\end{con}

We wish to emphasize that in the original conjecture of Deligne, 
the numbers $c^{\pm}$ are {\it periods}, 
but in this paper they are just a couple of complex numbers in terms of which the 
critical values of the symmetric power $L$-functions can be expressed. 

It adds some clarity to write down explicitly the statement
of the conjecture for the $n$-th symmetric power,
in the special cases $n=1,2,3,4$, and while doing so we also discuss about 
how much is known in these cases.

Let $m$ be a critical integer for $L_f(s, \varphi)$. Then 
Conjecture~\ref{conj:deligne} takes the form
\begin{equation}
\label{eqn:deligne-sym1}
L_f(m, \varphi) \sim
(2\pi i)^m c^{\pm}(\varphi),
\end{equation}
where $\pm = (-1)^m$.
In this context, the conjecture is known and is a theorem of Shimura 
\cite{shimura2} \cite{shimura3}. Shimura relates the required special 
values to quotients of certain Petersson inner products, whose rationality
properties can be studied.

Let $m$ be a critical integer for $L_f(s, {\rm Sym}^2 \varphi)$. Then
Conjecture~\ref{conj:deligne} takes the form
\begin{equation}
\label{eqn:deligne-sym2}
L_f(m, {\rm Sym}^2 \varphi) \sim
\left\{\begin{array}{ll}
(2\pi i)^{2m} (c^+(\varphi)c^-(\varphi)) \delta(\omega) &
\mbox{if m is even,} \\
(2\pi i)^m (c^+(\varphi)c^-(\varphi)) & \mbox{if m is odd.}
\end{array}\right.
\end{equation}
The conjecture is known in this case and is due to 
Sturm \cite{sturm1} \cite{sturm2}. Sturm uses an integral representation 
for the symmetric square $L$-function due to Shimura \cite{shimura1}.

Let $m$ be a critical integer for $L_f(s, {\rm Sym}^3 \varphi)$. Then
Conjecture~\ref{conj:deligne} takes the form
\begin{equation}
\label{eqn:deligne-sym3}
L_f(m, {\rm Sym}^3 \varphi) \sim
(2\pi i)^{2m} c^{\pm}(\varphi)^3 c^{\mp}(\varphi)\delta(\omega),
\end{equation}
where $\pm = (-1)^m$.
The conjecture is known in this case and is due to Garrett and Harris 
\cite{garrett-harris}. The main thrust of that paper is 
to prove a theorem on the special values of certain triple product 
$L$-functions $L(s, \varphi_1 \times \varphi_2 \times \varphi_3)$. 
Deligne's conjecture for motivic $L$-functions predicts the 
special values of such
triple product $L$-functions, for which an excellent reference is 
Blasius \cite{blasius1}. Via a standard argument, the case 
$\varphi_1 = \varphi_2 = \varphi_3 = \varphi$, 
gives the special values of the 
symmetric cube $L$-function for $\varphi$. 
This was reproved by 
Kim and Shahidi \cite{kim-shahidi-israel} emphasizing finiteness of 
these $L$-values which follows from their earlier work 
\cite{kim-shahidi-annals1}.

Let $m$ be a critical integer for $L_f(s, {\rm Sym}^4 \varphi)$. Then
Conjecture~\ref{conj:deligne} takes the form
\begin{equation}
\label{eqn:deligne-sym4}
L_f(m, {\rm Sym}^4 \varphi) \sim \left\{\begin{array}{ll}
(2\pi i)^{3m} (c^+(\varphi)c^-(\varphi))^3 \delta(\omega)^3 &
\mbox{if m is even,} \\
(2\pi i)^{2m}(c^+(\varphi)c^-(\varphi))^3 \delta(\omega)  &
\mbox{if m is odd.}
\end{array}\right.
\end{equation}
In general the conjecture is not known for higher ($n \geq 4$) symmetric
power $L$-functions. Although, if $\varphi$ is dihedral, then we have 
verified the conjecture for any symmetric power; see 
\S\ref{sec:dihedral}. 

We remark that a prelude to this conjecture was certain calculations 
made by Zagier \cite{zagier} wherein he showed that such a statement 
holds for the $n$-th symmetric power $L$-function, with $n \leq 4$, of 
the Ramanujan $\Delta$-function.

\subsection{Critical points}

As recalled above, an integer $m$ is
{\it critical} for any motivic $L$-function $L(s,M)$ if
both $L_{\infty}(s,M)$ and $L_{\infty}(1-s, M^{\vee})$ are regular at
$s=m$. For example, if $M = {\mathbb Z}(0) = H^*(\mbox{Point})$ 
be the trivial motive, then $L(s, M)$ is the Riemann zeta function 
$\zeta(s)$ \cite[\S 3.2]{deligne}. Then 
$L_{\infty}(s, M) = {\bf \pi}^{-s/2}\Gamma(s/2)$. It is an easy exercise 
to see that an integer $m$ is critical for $\zeta(s)$ if $m$ is an
even positive integer or an odd negative integer. 
More generally, as in Blasius \cite{blasius1}, one can 
calculate the critical points for any motivic $L$-function 
in terms of the Hodge numbers of the corresponding motive. 
For the specific $L$-functions at hand, namely the symmetric power 
$L$-functions, one explicitly knows the $L$-factors at infinity
\cite{moreno-shahidi} 
using which it is a
straightforward exercise to calculate the critical points. 
In the following two lemmas 
we record the critical points of the $n$-th
symmetric power $L$-function associated to a modular form $\varphi$.
(For more details see \cite{raghuram-shahidi}.)

\begin{lemma}
\label{lem:critical-sym-2r+1}
Let $\varphi$ be a primitive cusp form of weight $k$. The set of critical integers
for $L_f(s, {\rm Sym}^{2r+1} \varphi)$ is given by integers $m$ with
$$
r(k-1)+1 \leq m \leq (r+1)(k-1).
$$
\end{lemma}

\begin{lemma}
\label{lem:critical-sym-2r}
Let $\varphi$ be a primitive cusp form of weight $k$. The set of critical integers
for $L_f(s, {\rm Sym}^{2r} \varphi)$ is given below. 
{\small
\begin{enumerate}
\item If $r$ is odd and $k$ is even, then
$$
\{(r-1)(k-1)+1, (r-1)(k-1)+3,\dots, r(k-1);\  r(k-1)+1, r(k-1)+3,\dots, (r+1)(k-1)\}.
$$
\item If $r$ and $k$ are both odd, then
$$
\{(r-1)(k-1)+1, (r-1)(k-1)+3,\dots, r(k-1)-1;\  r(k-1)+2, r(k-1)+4,\dots, (r+1)(k-1)\}.
$$
\item If $r$ and $k$ are both even, then
$$
\{(r-1)(k-1)+2, (r-1)(k-1)+4,\dots, r(k-1)-1;\  r(k-1)+2, r(k-1)+4,\dots, (r+1)(k-1)-1\}.
$$
\item If $r$ is even and $k$ is odd, then
$$
\{(r-1)(k-1)+1, (r-1)(k-1)+3,\dots, r(k-1)-1;\  r(k-1)+2, r(k-1)+4,\dots, (r+1)(k-1)\}.
$$
\end{enumerate}}
\end{lemma}

\begin{rem}
\label{rem:easy}
{\rm Here are some easy observations based on the above lemmas. 
\begin{enumerate}
\item If $k=1$ then $L_f(s, {\rm Sym}^n \varphi)$ does not have any critical points
for any $n \geq 1$. In particular, this is the case if $\varphi$ is a cusp form which is 
tetrahedral, octahedral or icosahedral \cite{raghuram-shahidi}. 
\item If $k=2$ then $L_f(s, {\rm Sym}^n \varphi)$ has a critical point if and only 
if $n$ is not a multiple of $4$; further $L_f(s, {\rm Sym}^{2r+1} \varphi)$ has exactly 
one critical point $m=r+1$; and if $r$ is odd 
$L_f(s, {\rm Sym}^{2r} \varphi)$ has two critical points $r,r+1$. 
This applies in particular for symmetric power $L$-functions of elliptic curves. 
\item Let $m$ be a critical integer for 
$L_f(s, {\rm Sym}^{2r} \varphi)$. Then $m$ is even if and only if $m$ is to the
right of the center of symmetry.
\end{enumerate}

}
\end{rem}

\section{Dihedral calculations}
\label{sec:dihedral}

A primitive form $\varphi$ is said to be {\it dihedral} if the associated
cuspidal automorphic representation of 
${\rm GL}_2({\mathbb A}_{\mathbb Q})$, denoted $\pi(\varphi)$, is the 
automorphic induction of an id\`ele class character, say $\chi$, of a 
quadratic extension $K/{\mathbb Q}$. This is denoted as
$\pi(\varphi) = {\rm AI}_{K/{\mathbb Q}}(\chi)$.
(Since $\varphi$ is a holomorphic modular form,
in this situation, $K$ is necessarily an imaginary quadratic extension.)
In \cite{raghuram-shahidi} we give a proof of
Deligne's conjecture for the special values of symmetric power L-functions
for such dihedral forms using purely the language of L-functions. (See
the last paragraph of this section.)
In this section we summarize the main results of those calculations
while referring the reader to \cite{raghuram-shahidi} for all the proofs.

Recall from Remark~\ref{rem:easy} that if the weight $k=1$ then there
are no critical integers for $L_f(s, {\rm Sym}^n \varphi)$. It is easy
to see \cite{raghuram-shahidi} that if
$\pi(\varphi) = {\rm AI}_{K/{\mathbb Q}}(\chi)$ and some nonzero power of
$\chi$ is Galois invariant (under the Galois group of $K/{\mathbb Q}$)
then $k=1$. Hence we may, and henceforth shall,
assume that for every nonzero integer $n$, $\chi^n$ is not Galois 
invariant. The following lemma describes the isobaric decomposition of a 
symmetric power lifting of a dihedral cusp form.

\begin{lemma}
\label{lem:isobaric-dihedral}
Let $\chi$ be an id\`ele class character of an imaginary quadratic 
extension $K/{\mathbb Q}$; assume that 
$\chi^n$ is not Galois invariant for any nonzero integer $n$. Let
$\chi_{\mathbb Q}$ denote
the restriction of $\chi$ to the id\`eles of ${\mathbb Q}$.
Then we have
\begin{eqnarray*}
{\rm Sym}^{2r}({\rm AI}_{K/{\mathbb Q}}(\chi)) & = &
\boxplus_{a=0}^{r-1}{\rm AI}_{K/{\mathbb Q}}(\chi^{2r-a}\chi'^a)
\boxplus \chi_{\mathbb Q}^r,  \\
{\rm Sym}^{2r+1}({\rm AI}_{K/{\mathbb Q}}(\chi)) & = &
\boxplus_{a=0}^{r}{\rm AI}_{K/{\mathbb Q}}(\chi^{2r+1-a}\chi'^a),
\end{eqnarray*}
where $\chi'$ is the nontrivial Galois conjugate of $\chi$.
\end{lemma}

Note that every isobaric summand above is
either cuspidal or is one dimensional. This lemma
can be recast in terms of $L$-functions.
For an id\`ele class character $\chi$ of an imaginary quadratic
extension $K/{\mathbb Q}$, we let
$\varphi_{\chi}$ denote the primitive cusp form such that
$\pi(\varphi_{\chi}) = {\rm AI}_{K/{\mathbb Q}}(\chi)$. If
$\varphi_{\chi} \in S_k(N,\omega)$
then $\omega\omega_K = \chi_{\mathbb Q}$, where we make the obvious 
identification
of classical Dirichlet characters and id\`ele class characters 
of ${\mathbb Q}$, and $\omega_K$ denotes the quadratic id\`ele class
character of ${\mathbb Q}$ associated to $K$ via
global class field theory.

\begin{lemma}
\label{lem:dihedral-sym^n-l-fns}
The symmetric power $L$-functions of $\varphi_{\chi}$ decompose as
follows:
\small{
\begin{eqnarray*}
L_f(s, {\rm Sym}^{2r} \varphi_{\chi})
& = &
L_f(s-r(k-1), (\omega\omega_K)^r)
\prod_{a=0}^{r-1} L_f(s - a(k-1), \varphi_{\chi^{2(r-a)}}, \omega^a) \\
& = &
L_f(s-r(k-1), (\omega\omega_K)^r)
\prod_{a=0}^{r-1} L_f(s - a(k-1), \varphi_{\chi^{2(r-a)}},
(\omega \omega_K)^a). \\
L_f(s, {\rm Sym}^{2r+1} \varphi_{\chi})
& = &
\prod_{a=0}^{r} L_f(s - a(k-1), \varphi_{\chi^{2(r-a)+1}}, \omega^a) \\
& = &
\prod_{a=0}^{r} L_f(s - a(k-1), \varphi_{\chi^{2(r-a)+1}},
(\omega \omega_K)^a).
\end{eqnarray*}
}
\end{lemma}

We can now use the results of Shimura \cite{shimura2} \cite{shimura3} and
classical theorems on special values of abelian (degree $1$) $L$-functions
for the  factors on the right hand side of the above decompositions
to prove Deligne's conjecture on the special values of
$L_f(s, {\rm Sym}^n \varphi_{\chi})$. The proof is an extended exercise in 
keeping track of various constants after one has
related the periods of the cusp form $\varphi_{\chi^n}$ to the periods of
$\varphi_{\chi}$. We state this as the following theorem.

\begin{thm}[Period relations for dihedral forms]
\label{thm:period-relations}
For any positive integer $n$ we have the following relations:
\begin{enumerate}
\item $c^+(\varphi_{\chi^n}) \sim c^+(\varphi_{\chi})^n,$
\item $c^-(\varphi_{\chi^n}) \sim c^+(\varphi_{\chi})^n 
\mathfrak{g}(\omega_K),$
\end{enumerate}
where $\sim$ means up to an element of ${\mathbb Q}(\chi)$--the field 
generated
by the values of $\chi$, and
$\mathfrak{g}(\omega_K)$ is the Gauss sum of $\omega_K$.
\end{thm}

As mentioned in the introduction, Deligne's conjecture for dihedral forms is known. 
This is because Deligne's main conjecture is known if one considers only the motives
as those attached to abelian varieties and the category used is that defined
by using absolute Hodge cycles for morphisms. The theorem above can therefore be
deduced from the literature. However, the proof that we have is entirely in terms of 
$L$-functions: by considering the Rankin--Selberg $L$-function for 
$\varphi_{\chi^n} \times \varphi_{\chi}$. We also use some nonvanishing results
for $L$-functions in the course of proving this theorem.

\section{Representations with cohomology}
\label{sec:cohomology}

In the study of special values of $L$-functions, if the $L$-function at 
hand is associated to a cuspidal automorphic representation, then a 
standard assumption made on the representation is that it contributes to
cuspidal cohomology. This cohomology space admits a rational structure and 
certain periods are obtained 
by comparing this rational structure with a rational structure on the 
Whittaker model of the representation at hand. This approach to 
the study of special values is originally due to Harder \cite{harder}
and since then pursued by several authors and in particular by 
Mahnkopf \cite{mahnkopf2}.  

The purpose of this section, after setting up the context, 
is to record 
Theorem~\ref{thm:symmetric-cohomology} which says that the $n$-th 
symmetric power lift of a cohomological cusp form on ${\rm GL}_2$, if 
cuspidal, contributes to cuspidal cohomology of ${\rm GL}_{n+1}$. 
This theorem is essentially due to Labesse and Schwermer 
\cite{labesse-schwermer}. We then digress a little and discuss the 
issue of functoriality and a representation being cohomological.

\subsection{Cohomological representations of ${\rm GL}_n({\mathbb R})$}

In this section we set up the context of cohomological representations. 
This is entirely standard material; we refer the reader to 
Borel-Wallach \cite{borel-wallach} and Schwermer \cite{schwermer} for 
generalities on the cohomology of representations.

We let $G_n = {\rm GL}_n$ and $B_n$ be the standard Borel
subgroup of upper triangular matrices in $G_n$. 
Let $T_n$ be the diagonal torus in $G_n$ and $Z_n$ be the center of $G_n$. 
We denote by $X^+(T_n)$ the dominant (with respect to $B_n$) algebraic
characters of $T_n$. For $\mu \in X^+(T_n)$ let $(\rho_{\mu}, M_{\mu})$ 
be the irreducible representation of $G_n({\mathbb R})$ with highest 
weight $\mu$. The Lie algebra of $G_n({\mathbb R})$ will be denoted by 
$\mathfrak{g}_n$. 
We let $K_n = {\rm O}_n({\mathbb R})Z_n({\mathbb R})$ and  
$K_n^{\circ}$
be the topological connected component of the identity element in $K_n$. 

Let ${\rm Coh}(G_n, \mu)$ be the set of all cuspidal automorphic
representations $\pi = \otimes'_{p \leq \infty} \pi_p$ of
${\rm GL}_n({\mathbb A}_{\mathbb Q})$ such that
$$
H^*(\mathfrak{g}_n, K_n^{\circ} \ ;\ \pi_{\infty}\otimes \rho_{\mu}) \neq (0).
$$
By $H^*(\mathfrak{g}_n, K_n^{\circ};-)$ we mean 
relative Lie algebra cohomology. 
We recall the following from \cite[\S I.5.1]{borel-wallach}: 
Given a $(\mathfrak{g}_n, K_n)$ module $\sigma$, one can talk about 
$H^*(\mathfrak{g}_n, K_n^{\circ}; \sigma)$ as well as 
$H^*(\mathfrak{g}_n, K_n; \sigma)$. Note that $K_n/K_n^{\circ} \simeq 
{\mathbb Z}/2{\mathbb Z}$ acts 
on $H^*(\mathfrak{g}_n, K_n^{\circ}; \sigma)$ and by taking invariants
under this action we get $H^*(\mathfrak{g}_n, K_n; \sigma)$.

Observe that a global representation being cohomological is entirely a 
function of the representation at infinity. There are two very basic 
problems, one local and the other global, which has given rise to an 
enormous amount of literature on this theme.
\begin{enumerate}
\item The local problem is to classify all irreducible 
admissible representations $\pi_{\infty}$ 
of $G_n({\mathbb R})$ which are cohomological, i.e., 
$H^*(\mathfrak{g}_n, K_n^{\circ}; \pi_{\infty}\otimes \rho_{\mu}) \neq 
(0)$, and for 
such representations to actually calculate the cohomology spaces. 

\item The global problem is to construct global cuspidal  
representations whose representation at infinity is cohomological in the 
above sense. 
\end{enumerate}

The reader is referred to Borel-Wallach \cite{borel-wallach} as a 
definitive reference for the local problem. For the purposes of this 
article we discuss the solution of the local problem for tempered 
representations of ${\rm GL}_n({\mathbb R})$. To begin, 
we record a very simplified version of 
\cite[Theorem II.5.3]{borel-wallach} and \cite[Theorem 
II.5.4]{borel-wallach}.

\begin{thm}
\label{thm:discrete-series}
Let $G$ be a reductive Lie group. Let $K$ be a maximal compact subgroup 
adjoined 
with the center of $G$. Discrete series representations of $G$ (if they 
exist) are cohomological and have nonvanishing 
cohomology only in degree ${\rm dim}(G/K)/2$.
\end{thm}

We have suppressed any mention of the finite dimensional coefficients 
because Wigner's Lemma \cite[Theorem I.4.1]{borel-wallach} gives a 
necessary condition for the infinitesimal character, and nonvanishing 
cohomology of a representation pins down the finite dimensional 
representation. Here is a well known example illustrating this theorem.

\begin{exam}{\rm
\label{exam:gl2-cohomological}
Let $G = G_2({\mathbb R})$ and $K = K_2 = {\rm O}_2({\mathbb 
R})Z_2({\mathbb R})$. 
For any integer $l \geq 1$, we let
$M_l$ denote the irreducible representation of $G$ of dimension 
$l$ which is the $(l-1)$-th symmetric power of the standard two 
dimensional representation. 
Let $D_l$ be the discrete series representation of lowest weight $l+1$.  
(If we take a weight $k$ holomorphic cusp form then the representation 
at infinity is $D_{k-1}$.) The Langlands parameter of 
$D_l$ is ${\rm Ind}_{{\mathbb C}^*}^{W_{\mathbb R}}(\chi_l)$, where
$W_{\mathbb R}$ is the Weil group of ${\mathbb R}$, and 
$\chi_l$ is the character of ${\mathbb C}^*$ sending $z$ to 
$(z/|z|)^l$. The representation $D_l$ is cohomological; more precisely, we 
have
$$
H^q(\mathfrak{g},K; (D_l \otimes |\cdot |_{\mathbb R}^{-(l-1)/2})
 \otimes M_l) = 
\left\{\begin{array}{ll}
{\mathbb C} & \mbox{if $q = 1$,} \\
0 & \mbox{if $q \neq 1$.}
\end{array}\right.
$$
See \cite[Proposition I.4 (1)]{waldspurger} for instance. To compare our
notation to the notation therein, take $h = l+1$, $a=l-1$, $\epsilon 
=0$ and put $d = (h,a,\epsilon)$. Then our $M_l$ is the $r[d]$ of 
\cite{waldspurger} and our $D_l \otimes |\cdot |_{\mathbb R}^{-(l-1)/2}$
is the $\pi[d]$ of \cite{waldspurger}. See 
\cite[\S2.1]{labesse-schwermer} for an ${\rm SL}_2$ version of this 
example. 
}
\end{exam}

It is a standard fact that relative Lie algebra cohomology satisfies a  
K\"unneth rule \cite[\S I.1.3]{borel-wallach}. Using this one can 
see that if $G$ is a product of $m$ copies of ${\rm GL}_2({\mathbb R})$
then the representation $D_{l_1} \otimes \cdots \otimes D_{l_m}$ 
is, up to twisting by a suitable power of $|\cdot |_{\mathbb R}$,  
cohomological with respect to the finite dimensional coefficients 
$M_{l_1} \otimes \cdots \otimes M_{l_m}$.

We now recall, very roughly, a version of Shapiro's lemma 
for relative Lie algebra cohomology. 
Consider a parabolically induced representation. The cohomology of the
induced representation can be described in terms of the cohomology
of the inducing representation.
(See \cite[Theorem III.3.3, (ii)]{borel-wallach} for a precise 
formulation.)

We can now give a reasonably complete picture for tempered 
representations of ${\rm GL}_n({\mathbb R})$ which are cohomological. 
See Clozel \cite[Lemme 3.14]{clozel2}.
We follow the presentation in \cite[\S3.1]{mahnkopf2}.

Let $\mathcal{L}^+_0(G_n)$ stand for the set of all pairs
$(w,{\bf l}),$ with ${\bf l} = (l_1,\dots ,l_n) \in {\mathbb Z}^n$ such 
that 
$l_1 > \cdots > l_{[n/2]} > 0$ and $l_i = -l_{n-i+1}$, 
and $w \in {\mathbb Z},$ such that 
$$
w + {\bf l} \equiv 
\left\{\begin{array}{ll}
1 & \mbox{if $n$ is even,} \\
0 & \mbox{if $n$ is odd.}
\end{array}\right.
$$
This set $\mathcal{L}^+_0(G_n)$ will parametrize certain tempered 
representations defined as follows. For $(w,{\bf l}) \in 
\mathcal{L}^+_0(G_n)$, define the parabolically 
induced representation $J(w, {\bf l})$ by 
$$
J(w,{\bf l}) = {\rm Ind}_{P_{2,\dots,2}}^{G_n}
((D_{l_1}\otimes|\cdot|_{\mathbb R}^{w/2}) \otimes \cdots \otimes
(D_{l_{n/2}}\otimes|\cdot|_{\mathbb R}^{w/2}))
$$
if $n$ is even, and 
$$
J(w,{\bf l}) = {\rm Ind}_{P_{2,\dots,2,1}}^{G_n}
((D_{l_1}\otimes|\cdot|_{\mathbb R}^{w/2}) \otimes \cdots \otimes
(D_{l_{(n-1)/2}}\otimes|\cdot|_{\mathbb R}^{w/2}) \otimes
|\cdot|_{\mathbb R}^{w/2})
$$
if $n$ is odd. It is well known that, up to the twist 
$|\cdot|_{\mathbb R}^{w/2}$,
the representations $J(w,{\bf l})$ 
are irreducible tempered representations of $G_n$
\cite[\S 2]{knapp}. 

Now we describe the finite dimensional coefficients. Let $X_0^+(T_n)$ 
stand for all dominant integral weights $\mu = (\mu_1,\dots,\mu_n)$ 
satisfying the
purity condition that there is an integer $w$, called the weight of $\mu$,
such that $\mu_i + \mu_{n-i+1} = w$. The sets $\mathcal{L}^+_0(G_n)$
and $X_0^+(T_n)$ are in bijection via the map 
$(w,{\bf l}) \mapsto \mu = w/2 + {\bf l}/2 - \rho_n$ where 
$\rho_n$ is half the sum of positive roots for ${\rm GL}_n$. 
Let $w_n$ be the Weyl group element
of $G_n$ of longest length and let $\mu^{\vee} = -w_n\cdot \mu$. Then
$\rho_{\mu^{\vee}} \simeq (\rho_{\mu})^{\vee}$ is the contragredient of
$\rho_{\mu}$.

Assume that the pair $(w,{\bf l})$ corresponds to $\mu$ as above. 
Using Example~\ref{exam:gl2-cohomological}
on the cohomology of discrete series representations, and 
appealing to the K\"unneth rule and Shapiro's lemma as recalled above, one 
can conclude that
$$
H^q({\mathfrak g}_n, K_n^{\circ}; (J(w,{\bf l})\otimes {\rm sgn}^t) \otimes
M_{\mu^{\vee}}) = (0)
$$
unless the degree $q$ is in the so-called cuspidal range 
$b_n \leq q \leq t_n$, where the bottom degree $b_n$ is given by 
$$
b_n = 
\left\{\begin{array}{ll}
n^2/4 & \mbox{if $n$ is even,} \\
(n^2-1)/4 & \mbox{if $n$ is odd,}
\end{array}\right.
$$
and the top degree $t_n$ is given by
$$
t_n = 
\left\{\begin{array}{ll}
((n+1)^2-1)/4-1 & \mbox{if $n$ is even,} \\
(n+1)^2/4 -1 & \mbox{if $n$ is odd,}
\end{array}\right.
$$
and finally that the dimension of 
$H^q({\mathfrak g}_n,K_n ; (J(w,{\bf l})\otimes {\rm sgn}^t) 
\otimes M_{\mu^{\vee}})$ is $1$ if 
$q = b_n$ or $q = t_n$. The exponent $t$ of the sign character 
${\rm sgn}$ is in $\{0,1\}$. If $n$ is even, $t$ plays no role since
$J(w,{\bf l})\otimes {\rm sgn} = J(w,{\bf l})$. If $n$ is odd, $t$ is 
determined by the weight of $\mu$ and the parity of $(n-1)/2$, due to
considerations of central character (Wigner's lemma).

To complete the picture one notes that, given $M_{\mu}$,  
there is, up to twisting by the sign character, 
only one irreducible, unitary (up to twisting by  
$|\cdot |_{\mathbb R}^{-w/2}$), generic representation 
with nonvanishing cohomology with respect to $M_{\mu}$ and this 
representation is a suitable $J(w,{\bf l})$. 
(See \cite[\S 3.1.3]{mahnkopf2}.) 

\begin{rem}
{\rm 
\label{rem:j(w,l)}
Let $\pi$ be a cohomological cuspidal algebraic 
(\cite[\S1.2.3]{clozel2}) automorphic 
representation of $G_n({\mathbb A}_{\mathbb Q})$ 
then the representation $\pi_{\infty}$ at infinity has to be a 
$J(w,{\bf l})$ for some $(w,{\bf l}) \in \mathcal{L}_0^+(G_n)$. 
This can be seen as follows. Since $\pi$ is cuspidal and algebraic, by the 
purity lemma \cite[Lemme 4.9]{clozel2}, we get that the 
parameter of $\pi_{\infty}$ is pure. Since it is cohomological the 
finite dimensional coefficients has a highest weight $\mu$ which is 
also pure, i.e., $\mu \in X_0^+(T_n)$. 
Further, $\pi_{\infty}$ being generic and 
essentially unitary implies that it is a $J(w,{\bf l})$ as above. 
}
\end{rem}

\begin{exam}{\rm 
\label{exam:gl2}
To illuminate this picture we work through the above recipe for the case 
of a holomorphic cusp form. (We use the notation introduced in 
Example~\ref{exam:gl2-cohomological} and the previous sections.)
Let $\varphi \in S_k(N,\omega)$ and consider the 
cuspidal automorphic representation $\pi = \pi(\varphi)\otimes |\cdot|^s$. 
Then $\pi_{\infty} = \pi(\varphi)_{\infty} \otimes |\cdot|_{\mathbb R}^s
= D_{k-1} \otimes |\cdot|_{\mathbb R}^s.$ 
\begin{enumerate}
\item If $k$ is even, then the representation $\pi_{\infty}$ is a 
$J(w, {\bf l})$ exactly when $w = 2s \in {\mathbb Z}$ and $w + k-1$ is 
odd. Hence $s \in {\mathbb Z}$ and $\pi_{\infty} = J(2s, (k-1, -(k-1)))$.
The corresponding dominant weight $\mu$ is $(s + (k-2)/2, s - (k-2)/2)$. 
The representation $M_{\mu^{\vee}}$ is $M_{k-1} \otimes ({\rm 
det})^{-s-(k-2)/2}$.
(For a dominant weight $\mu = (\mu_1,\mu_2) \in {\mathbb Z}^2$
the rational representation $M_{\mu^{\vee}}$ is 
$M_{\mu_1-\mu_2+1} \otimes ({\rm det})^{-\mu_1}$.) Using the fact that 
${\rm det} = {\rm sgn} \otimes |\cdot |_{\mathbb R}$ and that 
$D_{k-1} \otimes {\rm sgn} = D_{k-1}$, we get 
\begin{eqnarray*}
\pi_{\infty} \otimes M_{\mu^{\rm \vee}} & = &
(D_{k-1} \otimes |\cdot|_{\mathbb R}^s) \otimes 
(M_{k-1} \otimes ({\rm det})^{-s-(k-2)/2}) \\
& = & 
(D_{k-1} \otimes |\cdot|_{\mathbb R}^{-(k-2)/2}) \otimes M_{k-1},
\end{eqnarray*}
which has nontrivial $(\mathfrak{g},K)$-cohomology (see
Example~\ref{exam:gl2-cohomological}).

\item If $k \geq 3$ is odd, then 
$\pi_{\infty}$ is a
$J(w, {\bf l})$ exactly when $w = 2s$ is an 
odd integer. Letting $s = 1/2 + r$, with $r \in  {\mathbb Z}$, we have 
$\pi_{\infty} = J(2r+1, (k-1, -(k-1)))$.
The corresponding dominant weight $\mu$ is $(r + (k-1)/2, r+1-(k-1)/2)$.
The representation $M_{\mu^{\vee}}$ is 
$M_{k-1} \otimes ({\rm det})^{-r-(k-1)/2}$.
In this case we get 
\begin{eqnarray*}
\pi_{\infty} \otimes M_{\mu^{\rm \vee}} & = &
(D_{k-1} \otimes |\cdot|_{\mathbb R}^{1/2+r}) \otimes
(M_{k-1} \otimes ({\rm det})^{-r-(k-1)/2}) \\
& = &
(D_{k-1} \otimes |\cdot|_{\mathbb R}^{-(k-2)/2}) \otimes M_{k-1},
\end{eqnarray*}
which has nontrivial $(\mathfrak{g},K)$-cohomology as mentioned before.
We have excluded the case $k=1$, because, firstly, the representation at 
infinity is not cohomological, and secondly, any symmetric power 
$L$-function of a weight one form has no critical points. 
\end{enumerate}
We finally remark that in both cases, the condition $\pi_{\infty}$ 
being a $J(w,{\bf l})$ is exactly the condition which ensures that 
the representation $\pi = \pi(\varphi)\otimes |\cdot |^s$ is 
regular algebraic in the sense of 
Clozel \cite[\S1.2.3 and \S3.4]{clozel2}.
}
\end{exam}

\subsection{Functoriality and cohomological representations}

Now we turn to the global problem, 
namely, to construct a cuspidal automorphic representation whose 
representation at infinity is cohomological. The specific theorem 
we are interested in is the following.

\begin{thm}
\label{thm:symmetric-cohomology}
Let $\varphi \in S_k(N,\omega)$ with $k \geq 2$. 
Let $n \geq 1$. Assume that
${\rm Sym}^n(\pi(\varphi))$ is a cuspidal representation of
${\rm GL}_{n+1}({\mathbb A}_{\mathbb Q})$. Let 
$$
\Pi = {\rm Sym}^n(\pi(\varphi)) \otimes \xi \otimes |\cdot |^s
$$
where $\xi$ is any id\`ele class character such that 
$\xi_{\infty} = {\rm sgn}^{\epsilon}$, with $\epsilon \in \{0,1\}$, 
and $|\cdot |$ is the ad\`elic norm. We suppose that $s$ and $\epsilon$ 
satisfy: 
\begin{enumerate}
\item If $n$ is even, then let $s \in {\mathbb Z}$ and 
$\epsilon \equiv n(k-1)/2 \pmod{2}$. 
\item If $n$ is odd then, we let $s \in {\mathbb Z}$ if k is even,
and we let $s \in 1/2+{\mathbb Z}$ if k is odd. We impose no condition 
on $\epsilon$. 
\end{enumerate}
Then 
$\Pi \in {\rm Coh}(G_{n+1}, \mu^{\vee})$ 
where $\mu \in X^+_0(T_{n+1})$ is given by 
$$
\mu = \left(
\frac{n(k-2)}{2}+s, \frac{(n-2)(k-2)}{2}+s,\dots, \frac{-n(k-2)}{2}+s
\right) = (k-2)\rho_{n+1}+s.
$$
(Recall that $\rho_{n+1}$ is half the sum of positive roots of ${\rm 
GL}_{n+1}$.)
In other words, the representation 
${\rm Sym}^n(\pi(\varphi)) \otimes \xi \otimes |\cdot |^s$, with $\xi$ and 
$s$ as above, contributes to 
the cohomology of the locally symmetric space 
${\rm GL}_{n+1}({\mathbb Q})\backslash {\rm GL}_{n+1}({\mathbb 
A}_{\mathbb Q})/
K_f K_{n+1, \infty}^{\circ}$ with coefficients in the local 
system determined by 
$\rho_{\mu^{\vee}}$, where $K_f$ is a deep enough open compact subgroup of 
${\rm GL}_{n+1}({\mathbb A}_{{\mathbb Q},f})$. (Here ${\mathbb 
A}_{{\mathbb Q},f}$ denotes the finite ad\`eles of ${\mathbb Q}$.) 
\end{thm}

\begin{proof}
See Labesse--Schwermer \cite[Proposition 5.4]{labesse-schwermer} for 
an ${\rm SL}_n$-version of this theorem. 
When $k=2$, the theorem has also been observed by Kazhdan, Mazur and 
Schmidt \cite[pp.99]{kazhdan-mazur-schmidt}.

We sketch the details in the case when $n = 2r$ is even (the case when 
$n$ is odd being absolutely similar.) The proof follows by observing that the 
representation at infinity of 
${\rm Sym}^n(\pi(\varphi))$ is the representation of 
${\rm GL}_{n+1}({\mathbb R})$ whose Langlands parameter is 
${\rm Sym}^n({\rm Ind}_{\mathbb C^*}^{W_{\mathbb R}}(\chi_{k-1}))$ 
where $\chi_{k-1}(z) = (z/|z|)^{k-1}$. It is a pleasant exercise to
calculate a symmetric power of a two dimensional induced representation,
after doing which one gets that the representation 
$\Pi_{\infty}$ is given by 
\begin{eqnarray*}
\Pi_{\infty} & = &
{\rm Ind}_{P_{2,\dots,2,1}}^{G_{n+1}}
(D_{2r(k-1)}\otimes \cdots \otimes D_{2(k-1)} \otimes 
{\rm sgn}^{r(k-1)}) \otimes \xi_{\infty} \otimes |\cdot |_{\mathbb R}^s \\
& = &
{\rm Ind}_{P_{2,\dots,2,1}}^{G_{n+1}}
(D_{2r(k-1)}\otimes \cdots \otimes D_{2(k-1)} \otimes
{\rm sgn}^{r(k-1)+\epsilon}) \otimes |\cdot |_{\mathbb R}^s.
\end{eqnarray*}

We deduce that $\Pi_{\infty}$ is a $J(w,{\bf l})$ (which, as mentioned 
before, is equivalent to $\Pi$ being regular algebraic) exactly when 
$w = 2s \in {\mathbb Z}$, $r(k-1)+\epsilon$ is even, 
$$
{\bf l} = (2r(k-1),\dots,2(k-1),0,-2(k-1),\dots,-2r(k-1)) 
= 2(k-1)\rho_{n+1},
$$ and 
$w + {\bf l}$ is even which implies that $w$ is even. These conditions 
are satisfied by the hypothesis in the theorem. The weight $\mu$ 
is determined by $\mu = w/2+{\bf l}/2-\rho_{n+1}$ and the first 
part of the theorem follows 
from the discussion in the previous section. 

Finally, the relation with the cohomology of locally symmetric 
spaces follows as in \cite[\S1]{labesse-schwermer} or 
\cite[\S3.2]{mahnkopf2}.
\end{proof}

\begin{cor}
Let $\varphi \in S_k(N,\omega)$. Assume that $k \geq 2$ and that $\varphi$ is not dihedral. 
Then, up to twisting by a quadratic character, ${\rm Sym}^n(\pi(\varphi))$
for $n=2,3,4$, contributes to cuspidal cohomology. 
\end{cor}

\begin{proof}
Other than the above theorem one appeals to the main result of \cite{kim-shahidi-duke}.
\end{proof}

One might view this theorem as an example of the possible dictum that 
a functorial lift of a cohomological representation is cohomological. 
This has been used in many instances to construct global 
representations which contribute to cuspidal cohomology.  However, 
there are several instances where a functorial lift of a cohomological
representation is not cohomological. On the other hand, it is an 
interesting question to ask if the converse of the above dictum is true, 
namely, if a lift is cohomological, then whether the preimage, is {\it a fortiori}
cohomological? We end the section by a series of examples illustrating some of these
principles. But before doing so, we would like to draw the reader's attention to a conjecture
of Clozel \cite[\S1]{clozel1} which is motivated by the ideas of 
Labesse and Schwermer \cite{labesse-schwermer}. It roughly states that given a tempered 
cohomological representation at infinity, one can find a global cuspidal
automorphic representation whose representation at infinity is the given one.

\begin{exam}{\rm 
The following 
is a sampling of results--which by no means is to be considered 
exhaustive--where functoriality is used to produce cohomological 
representations. 
\begin{enumerate}

\item Labesse and Schwermer \cite{labesse-schwermer}
proved the existence of nontrivial cuspidal 
cohomology classes for ${\rm SL}_2$ and ${\rm SL}_3$ over any number field 
$E$ which contains a totally real number field $F$ such that $F = F_0 
\subset F_1 \subset \cdots \subset F_n = E$ with each each $F_{i+1}/F_i$ 
either a cyclic extension of prime degree or a non-normal cubic extension. 
The functorial lifts used were base change for ${\rm GL}_2$ and the 
symmetric square lifting of Gelbart and Jacquet. This was 
generalized for ${\rm SL}_n$ over $E$, in conjunction with Borel 
\cite{borel-labesse-schwermer}; with the additional input of base change 
for ${\rm GL}_n$. 

\item Motivated by \cite{labesse-schwermer}, Clozel \cite{clozel1} used 
automorphic induction and proved the existence of nontrivial cuspidal 
cohomology classes for ${\rm SL}_{2n}$ over any number field.

\item Rajan \cite{rajan}, also motivated by \cite{labesse-schwermer}, 
proved the existence of nontrivial cuspidal cohomology classes for 
${\rm SL}_1(D)$ for a quaternion division algebra $D$ over a number field 
$E$, with $E$ being
an extension of a totally real number field $F$ with solvable 
Galois closure. Other than base change, he used the Jacquet--Langlands 
correspondence.

\item Ash and Ginzburg \cite[\S 4]{ash-ginzburg} have commented on 
a couple of examples of cuspidal cohomology classes for 
${\rm GL}_4$ over ${\mathbb Q}$. The first is 
by lifting from ${\rm GSp}_4$ to ${\rm GL}_4$
a weight $3$ Seigel modular form. The second is to use automorphic 
induction from ${\rm GL}_2$ over a quadratic extension. 

\item Ramakrishnan and Wang \cite{ramakrishnan-wang} used the 
lifting from ${\rm GL}_2 \times {\rm GL}_3 \to {\rm GL}_6$, due to Kim and 
Shahidi, to construct cuspidal cohomology classes of ${\rm GL}_6$ over 
${\mathbb Q}$. 
\end{enumerate}
In almost all the above works, functoriality is used to construct 
cuspidal representations, and in doing so, one exercises some control over
the representations at infinity to arrange for them to be cohomological. 
}
\end{exam}

\begin{exam}
\label{exam:counterexample}{\rm 
We construct an example to show that a functorial lift of a cohomological
representation need not be cohomological. 
For an even integer $k$, 
take two weight $k$ holomorphic cusp forms $\varphi_1$
and $\varphi_2$, and let $\pi_i = \pi(\varphi_i)$ for $i =1,2$. 
By Example~\ref{exam:gl2} we have that both $\pi_1$ and $\pi_2$ are 
cohomological representations. 
Put $\Pi = \pi_1 \boxtimes \pi_2$ (see Ramakrishnan \cite{dinakar1}).
Choose the forms $\varphi_1$ and $\varphi_2$ such that $\Pi$ is cuspidal; 
this can be arranged by taking exactly one of them to be dihedral, or by 
arranging that $\pi_1$ is not $\pi_2 \otimes \chi$ for any character 
$\chi$, by virtue of \cite[Theorem 11.1]{dinakar2}. It is easy to see 
that $\Pi_{\infty}$ is given by
$$
\Pi_{\infty} = {\rm Ind}_{P_{2,1,1}}^{G_4({\mathbb R})}
(D_{2(k-1)} \otimes {\rm sgn} \otimes 1\!\!1),
$$
where $1\!\!1$ is the trivial representation of ${\mathbb R}^*$. Observe
that $\Pi_{\infty}$ is not a $J(w,{\bf l})$ and hence is not cohomological
by applying Remark~\ref{rem:j(w,l)}. (Note that $\Pi$, as it stands, is 
not algebraic, but we can replace $\Pi$ by 
$\pi_1 \stackrel{T}{\boxtimes} \pi_2$ (see \cite[Definition 
1.10]{clozel2})
and make it algebraic; this replaces 
$\Pi_{\infty}$ by $\Pi_{\infty}\otimes |\cdot |_{\mathbb R}^{1/2}$.)
However, note that if we took 
$\varphi_1$ and $\varphi_2$ to be in general position (unequal weights) 
then the lift $\Pi$ would be cohomological. 
Similarly, it is possible to
construct such an example for the lifting from 
${\rm GL}_2 \times {\rm GL}_3$ to ${\rm GL}_6$. 
}
\end{exam}

\begin{exam}
\label{exam:counterexample2}
{\rm We would like to mention that in the converse direction the 
${\rm GL}_2 \times {\rm GL}_2$ to ${\rm GL}_4$ lifting is well behaved. 
Now let $\varphi_i$ have weight $k_i \geq 1$, for $i=1,2$, and assume 
without loss of generality that $k_1 \geq k_2$. With $\Pi = \pi(\varphi_1) 
\boxtimes \pi(\varphi_2)$ we have 
$$
\Pi_{\infty} = {\rm Ind}_{P_{2,2}}^{G_4({\mathbb R})}
(D_{k_1+k_2-2} \otimes D_{k_1-k_2}).
$$
Suppose $\Pi_{\infty}$ is cohomological, i.e., is a $J(w,{\bf l})$, 
then we would have $k_1+k_2-2 > k_1-k_2 > 0$, which  
implies that $k_1> k_2 \geq 2$, and hence both $\pi(\varphi_1)$ and 
$\pi(\varphi_2)$ are cohomological. 
}\end{exam}

\begin{exam}
\label{exam:counterexample3}
{\rm It seems to be well known that given a cuspidal representation $\pi$ of 
${\rm GL}_4({\mathbb A}_{\mathbb Q})$, its exterior square lift $\wedge^2(\pi)$,
which by Kim \cite{kim-jams} is an automorphic representation of 
${\rm GL}_6({\mathbb A}_{\mathbb Q})$, is never 
cohomological. 
}
\end{exam}

\section{Special values of $L$-functions of ${\rm GL}_n$: The work of 
Mahnkopf}
\label{sec:mahnkopf}

\subsection{General remarks on functoriality and special values}
This section is a summary of some recent results due to Joachim Mahnkopf
\cite{mahnkopf1} \cite{mahnkopf2}. In this work he proves certain 
special values theorems for the standard $L$-functions of cohomological
cuspidal automorphic representations of ${\rm GL}_n$. 
In principle one can appeal to functoriality and 
this work of Mahnkopf to prove new 
special values theorems. For example, given a cusp form 
$\varphi \in S_k(N,\omega)$, let $\pi(\varphi)$ denote the 
cuspidal automorphic representation of 
${\rm GL}_2({\mathbb A}_{\mathbb Q})$.
Functoriality 
predicts the existence of an automorphic
representation ${\rm Sym}^n(\pi(\varphi))$ of 
${\rm GL}_{n+1}({\mathbb A}_{\mathbb Q})$. (See \S\ref{sec:functorial}.)
Then it is easy to check that
$$
L(s, {\rm Sym}^n(\pi(\varphi)) = 
L(s + n(k-1)/2, {\rm Sym}^n \varphi),
$$
where the left hand side is the standard $L$-function of 
${\rm Sym}^n(\pi(\varphi))$. Using Mahnkopf's work for the
function on the left, one can hope to prove a special values
theorem for the function on the right. This is fine in principle,
but there are several obstacles to overcome before it can be made to work.

\subsection{The main results of Mahnkopf \cite{mahnkopf2}} 
Let $\mu \in X_0^+(T_n)$ and 
let $\pi \in {\rm Coh}(G_n, \mu)$. We let $L(s,\pi) = \prod_{p \leq 
\infty} L(s, \pi_p)$ be the standard $L$-function attached to $\pi$. Any 
character $\chi_{\infty}$ of ${\mathbb R}^*$ is of the form $\chi_{\infty} 
= \epsilon_{\infty}|\cdot |^m$ for a complex number $m$. We say 
$\chi_{\infty}$ is critical for $\pi_{\infty}$ if
\begin{enumerate}
\item $m \in 1/2 + {\mathbb Z}$ if $n$ is even, and $m \in {\mathbb Z}$ if 
$n$ is odd; and
\item $L(\pi_{\infty} \otimes \chi_{\infty}, 0)$ and 
$L(\pi_{\infty}^{\vee} \otimes \chi_{\infty}^{-1}, 1)$ are regular values.
\end{enumerate}
We say 
$\chi : {\mathbb Q}^*\backslash {\mathbb A}_{\mathbb Q}^{\times} 
\to {\mathbb C}^*$ is critical for 
$\pi$ if $\chi_{\infty}$ is critical for 
$\pi_{\infty}$. Let ${\rm Crit}(\pi)$ stand for all such characters $\chi$ 
which are critical for $\pi$. Let ${\rm Crit}(\pi)^{\leq}$ stand for all 
$\chi \in {\rm Crit}(\pi)$ such that if $\chi_{\infty} = 
\epsilon_{\infty}|\cdot |^m$ then $m \leq (1-{\rm wt}(\mu))/2$.

Let $\pi \in {\rm Coh}(G_n, \mu)$ and let $\chi \in {\rm Crit}(\pi)$. Let
$\chi_{\infty} = \epsilon_{\infty}|\cdot |^m$.
Given $\mu = (\mu_1,\dots,\mu_n) \in X^+(T_n)$ choose a 
$\lambda = (\lambda_1,\dots,\lambda_{n-1}) \in X^+(T_{n-1})$ such that 
\begin{enumerate}
\item $\mu_1 \geq \lambda_1 \geq \mu_2 \geq \cdots \geq \lambda_{n-1} 
\geq \mu_n$; and
\item $\lambda_{n/2}=-m+1/2$ if $n$ is even, and $\lambda_{(n+1)/2} = -m$ 
if $n$ is odd. 
\end{enumerate}
Proposition 1.1 of \cite{mahnkopf1} says that such a $\lambda$ exists. Let 
$P$ be the standard parabolic subgroup of $G_{n-1}$ of type $(n-2,1)$ and let 
$W^P$ be a system of representatives for $W_{M_P}\backslash W_{G_{n-1}}$. 
Let $\widehat{w} \in W^P$ be given by
$$
\widehat{w} = 
\left(\begin{array}{cccccccc}
1 & 2 & \cdots & \left[\frac{n}{2}\right]-1 & \left[\frac{n}{2}\right] & 
\left[\frac{n}{2}\right]+1 & \cdots & n-1 \\
1 & 2 & \cdots & \left[\frac{n}{2}\right]-1 & n-1 & 
\left[\frac{n}{2}\right] & \cdots & n-2
\end{array}\right).
$$
Define the weight $\mu'= (\widehat{w}(\lambda + \rho_{n-1}) - 
\rho_{n-1})|_{T_{n-2}} \in X^+(T_{n-2}) $ where $T_{n-2}$ is embedded in 
$T_{n-1}$ as $t \mapsto {\rm diag}(t,1)$.

\begin{thm}[Theorem 5.4 in Mahnkopf \cite{mahnkopf2}]
\label{thm:mahnkopf-1}
Let $\mu \in X^+_0(T_n)$ be regular and let $\pi \in {\rm Coh}({\rm GL}_n, 
\mu^{\vee})$.
Let $\mu' \in X^+(T_{n-2})$ be as above and 
$\pi' \in {\rm Coh}({\rm GL}_{n-2}, \mu')$; if $n$ is odd then $\pi'$ has 
to satisfy
a parity condition. We have 
\begin{enumerate}
\item ${\rm Crit}(\pi)^{\leq} \subset {\rm Crit}(\pi')^{\leq}$.
\item Let $\chi \in {\rm Crit}(\pi)^{\leq}$, with 
$\chi_{\infty} = \epsilon_{\infty}|\cdot |^m$. There exists a collection of
complex numbers 
$\Omega(\pi,\pi',\epsilon_{\infty}) \in {\mathbb C}^*/{\mathbb 
Q}(\pi){\mathbb Q}(\pi')$
such that for any finite extension $E/{\mathbb Q}(\pi){\mathbb Q}(\pi')$ 
the tuple
$\{ \Omega(\pi,\pi',\epsilon_{\infty}) \}_{\sigma \in {\rm 
Hom}(E,{\mathbb C})} 
\in (E\otimes {\mathbb C})^*/ ({\mathbb Q}(\pi){\mathbb Q}(\pi'))^*$ is 
well defined. 
There exists a complex number $P_{\mu}(m)$,
depending only on $\mu$ and $m$, subject to Assumption~\ref{ass:mahnkopf} 
below,  
such that for all $\sigma \in {\rm Aut}({\mathbb C}/{\mathbb Q})$ 
and almost all $\chi$ as above, we have
$$
\left(
\frac{\mathfrak{g}(\chi) \mathfrak{G}(\eta)P_{\mu}(m)}
{\Omega(\pi,\pi',\epsilon_{\infty})}
\frac{L(\pi\otimes\chi\eta, 0)}
{L(\pi'^{\vee} \otimes \chi, 0)}\right)^{\sigma} = 
\frac{\mathfrak{g}(\chi^{\sigma})\mathfrak{G}(\eta^{\sigma})P_{\mu}(m)}
{\Omega(\pi^{\sigma},\pi'^{\sigma},\epsilon_{\infty})}
\frac{L(\pi^{\sigma}\otimes\chi^{\sigma}\eta^{\sigma}, 0)}
{L((\pi'^{\vee})^{\sigma} \otimes \chi^{\sigma}, 0)},
$$
where $\eta$ is a certain auxiliary character and 
$\mathfrak{G}(\eta)$ a certain product of Gauss sums associated to
$\eta$.
\end{enumerate}
\end{thm}

The above theorem is valid only under the following assumption. 

\begin{ass}
\label{ass:mahnkopf}
$
P_{\mu}(m) \neq 0. 
$
\end{ass}

The quantity $P_{\mu}(m)$ is the value at $s=1/2$ of an archimedean 
Rankin--Selberg integral attached to certain cohomological choice of 
Whittaker functions. Mahnkopf proves a necessary condition for this 
nonvanishing assumption \cite[\S6]{mahnkopf2}. 
At present this seems to be a serious limitation of this technique. 
It is widely believed that 
this assumption is valid and it has shown up in 
several other works based on the same, or at any rate similar, techniques. 
See for instance Ash--Ginzburg \cite{ash-ginzburg}, 
Kazhdan-Mazur-Schmidt 
\cite{kazhdan-mazur-schmidt} and Harris \cite{harris}.
{\it It is an important technical problem to be able to prove this nonvanishing 
hypothesis.} 

The proof of the above theorem combines both the Langlands--Shahidi 
and the Rankin--Selberg methods of studying $L$-functions. One considers
the pair of representations
$
\pi \times {\rm Ind}_P^{G_{n-1}}(\pi' \otimes \chi)
$
of $G_n({\mathbb A}_{\mathbb Q}) \times G_{n-1}({\mathbb A}_{\mathbb Q})$ 
and carefully 
chooses a cusp form $\phi \in \pi$ and an 
Eisenstein series $\mathcal{E}$ corresponding to a section in ${\rm 
Ind}_P^{G_{n-1}}(\pi' \otimes \chi)$. To this pair $(\phi,\mathcal{E})$ a 
certain Rankin--Selberg type zeta integral \cite[2.1.2]{mahnkopf2},
which has a cohomological interpretation, computes the quotient of 
$L$-functions appearing in the theorem.

The theorem roughly 
says that the special values of a standard $L$-function for ${\rm GL}_n$
are determined in terms of those of a standard $L$-function for ${\rm 
GL}_{n-2}$. 
This descent process terminates since we know the special values of 
$L$-functions for 
${\rm GL}_1$ and ${\rm GL}_2$, and we get the following 
theorem; see \cite[\S 5.5]{mahnkopf2} for making the right choices 
in the induction on $n$.

\begin{thm}[Theorem A in Mahnkopf \cite{mahnkopf2}]
\label{thm:mahnkopf-2}
Assume that $\mu \in X^+_0(T_n)$ is regular and let $\pi \in {\rm 
Coh}(G_n,\mu^{\vee})$. Let $\chi \in {\rm Crit}(\pi)^{\leq}$. To 
$\pi$ and $\chi_{\infty}$ is attached  
$\Omega(\pi,\chi_{\infty}) \in {\mathbb C}$ such that for all 
but finitely many such $\chi$ we have 
$$
\left(
\frac{\mathfrak{g}(\chi)^{[n/2]}\mathfrak{G}(\eta)}
{\Omega(\pi, \chi_{\infty})}
L(\pi\otimes\chi\eta, 0)
\right)^{\sigma} 
=
\frac{\mathfrak{g}(\chi^{\sigma})^{[n/2]}
\mathfrak{G}(\eta^{\sigma})}
{\Omega(\pi^{\sigma},\chi^{\sigma}_{\infty})}
L(\pi^{\sigma}\otimes\chi^{\sigma}\eta^{\sigma}, 0),
$$
where $\eta$ is a certain auxiliary character and
$\mathfrak{G}(\eta)$ a certain product of Gauss sums associated to
$\eta$. Moreover, write $\chi_{\infty} = \epsilon_{\infty}'|\cdot 
|_{\infty}^l$ and set $\epsilon(\chi_{\infty}) = \epsilon_{\infty}'{\rm 
sgn}^l$. There are periods $\Omega_{\epsilon}(\pi) \in {\mathbb C}^*$ 
if $n$ is even, and $\Omega(\pi) \in {\mathbb C}^*$ if $n$ is odd, and a 
collection $P_{\mu}^l \in {\mathbb C}$, such that 
$\Omega(\pi, \chi_{\infty}) = P_{\mu}^l\Omega(\pi)$ if $n$ is odd, and 
$\Omega(\pi, \chi_{\infty}) = 
P_{\mu}^l\Omega_{\epsilon(\chi_{\infty})}(\pi)$ if $n$ is even. 

\end{thm}

Note that Theorem~\ref{thm:mahnkopf-2}, since it uses  
Theorem~\ref{thm:mahnkopf-1}, also depends on 
Assumption~\ref{ass:mahnkopf}.

\section{A conjecture on twisted $L$-functions}

The periods $c^+$ and $c^-$ which appear in Deligne's conjecture are 
motivically defined. (See Deligne \cite[(1.7.2)]{deligne}.) On the other 
hand, the periods which appear in the work of Harder, and 
also Mahnkopf, have an entirely different origin, namely, 
they come by a comparison of rational structures on cuspidal cohomology 
on the one hand and a Whittaker model for the representation, on the 
other. See Harder \cite[p. 81]{harder} and Mahnkopf \cite[\S 
3.4]{mahnkopf2}. It is not at the moment clear how one might explicitly 
compare these different periods attached to the same object. (See also 
Remark (2) in Harder's paper \cite[p. 85]{harder}.)

However, one might ask if these different periods behave in the same 
manner under twisting. Here is a simple example to illustrate this. 
Let $\chi$ be an even Dirichlet character. Let $m$ be an even 
positive integer. Such an $m$ is critical for $L(s,\chi)$. It is well 
known \cite[Corollary VII.2.10]{neukirch} that 
$$
L_f(m,\chi) \sim_{{\mathbb Q}(\chi)} (2\pi i)^m \mathfrak{g}(\chi).
$$ 
By $\sim_{{\mathbb Q}(\chi)}$ we mean up to an element of the 
(rationality) field ${\mathbb 
Q}(\chi)$ generated by the values of $\chi$. 
Now let $\eta$ be possibly another even Dirichlet character. 
Applying the result to the character $\chi\eta$, and using 
\cite[Lemma 8]{shimura2}, we get
$$
L_f(m,\chi\eta)/L_f(m,\chi) \sim_{{\mathbb Q}(\chi){\mathbb Q}(\eta)} 
\mathfrak{g}(\eta).
$$
Observe that the {\it period}, namely the $(2\pi i)^m$, does not show up, 
and we have the relation that the special value of the twisted 
$L$-function and the original $L$-function differ, up to rational 
quantities, by the Gauss sum of the twisting character. 

Another example along these lines which follows from Shimura \cite{shimura3} 
is the following. Let $\varphi \in S_k(N,\omega)$ and let $\eta$ be an 
even Dirichlet character. For any integer $m$, with $1 \leq m \leq k-1$, 
we have 
$$ 
L_f(m,\varphi,\eta) \sim_{{\mathbb Q}(\varphi){\mathbb Q}(\eta)} 
\mathfrak{g}(\eta) 
L_f(m,\varphi).
$$ 
The point being that, in the above relation, the periods  
$c^{\pm}(\varphi)$ do not show up, and so the definition of these periods 
is immaterial. (One can rewrite this relation entirely in terms of periods 
of the associated motives and it takes the form 
$ c^{\pm}(M(\varphi) 
\otimes M(\eta)) \sim \mathfrak{g}(\eta)c^{\pm}(M(\varphi)), 
$ 
the notation being obvious.) 

Even if one cannot prove a precise theorem on 
special values of $L$-functions in terms of these--motivically or 
otherwise defined--periods, one can still hope to prove such period 
relations. Sometimes such period relations are sufficient for 
applications; see for instance Murty--Ramakrishnan \cite{kumar-dinakar} 
where such a period relation is used to prove Tate's conjecture in a 
certain case.

With this motivation, we formulate the following conjecture on the 
behavior of the special values of symmetric power $L$-functions under 
twisting by Dirichlet characters.

\begin{con}
\label{con:twisted}
Let $\varphi \in S_k(N,\omega)$ be a primitive form. Let $\eta$ be a 
primitive Dirichlet character. 
\begin{enumerate}
\item Suppose $\eta$ is even, i.e., $\eta(-1)=1$. Then the critical set 
for $L_f(s, {\rm Sym}^n\varphi, \eta)$ is the same as the critical set 
for $L_f(s, {\rm Sym}^n\varphi),$ and if $m$ is critical, then 
$$
L_f(m, {\rm Sym}^n\varphi, \eta) \sim 
\mathfrak{g}(\eta)^{\lceil (n+1)/2 \rceil}
L_f(m, {\rm Sym}^n\varphi),
$$ 
unless $n$ is even and $m$ is odd (to the left of center of symmetry), 
in which case we have
$$
L_f(m, {\rm Sym}^n\varphi, \eta) \sim \mathfrak{g}(\eta)^{n/2}
L_f(m, {\rm Sym}^n\varphi).
$$ 
\smallskip

\item Suppose $\eta$ is odd, i.e., $\eta(-1)=-1,$ and n is even. 
Then, if $m$ is critical  
for $L_f(s, {\rm Sym}^n\varphi, \eta),$ then 
either $m+1$ or $m-1$ is critical for $L_f(s, {\rm Sym}^n\varphi)$. 
For such an $m$ to the right of the center of symmetry we have 
$$
L_f(m, {\rm Sym}^n\varphi, \eta) \sim 
((2\pi i)^{\mp}\mathfrak{g}(\eta))^{n/2+1}
L_f(m \pm 1, {\rm Sym}^n\varphi),
$$ 
and if $m$ is to the left of the center of symmetry, we have
$$
L_f(m, {\rm Sym}^n\varphi, \eta) \sim 
((2\pi i)^{\mp}\mathfrak{g}(\eta))^{n/2}
L_f(m \pm 1, {\rm Sym}^n\varphi).
$$ 
\smallskip

\item Suppose $\eta$ is odd, i.e., $\eta(-1)=-1$, and n is odd. 
Then the critical set 
for $L_f(s, {\rm Sym}^n\varphi, \eta)$ is the same as the critical set 
for $L_f(s, {\rm Sym}^n\varphi)$. Let $k \geq 3$. If $m$ is 
critical for $L_f(s, {\rm Sym}^n\varphi, \eta),$ then either $m+1$ or 
$m-1$ is critical for $L_f(s, {\rm Sym}^n\varphi),$ and for such an $m$
$$
L_f(m, {\rm Sym}^n\varphi, \eta) \sim 
((2\pi i)^{\mp}\mathfrak{g}(\eta))^{(n+1)/2}
L_f(m \pm 1, {\rm Sym}^n\varphi).
$$ 
\end{enumerate}
In all the three cases $\sim$ means up to an element of 
${\mathbb Q}(\varphi){\mathbb Q}(\eta)$.
\end{con}

Now we elaborate on the heuristics on which we formulated the above 
conjecture. For $n=1$ and $n=2$ this is contained in 
the theorems of Shimura \cite{shimura2} \cite{shimura3} and 
Sturm \cite{sturm1} \cite{sturm2} respectively. 
For $n=3$, using results on triple product $L$-functions 
for which Blasius \cite{blasius1} is a convenient reference and 
using Garrett--Harris \cite[\S 6]{garrett-harris}, one can verify 
that the above conjecture is true. Further, for $n \geq 4$ and if 
$\varphi$ is 
dihedral, i.e., $\pi(\varphi) = {\rm AI}_{K/{\mathbb Q}}(\chi)$, then the 
conjecture follows by applying the known cases of $n=1,2$ to each 
summand 
in the isobaric decomposition in Lemma~\ref{lem:isobaric-dihedral}.  
Observe that the exponent $\lceil (n+1)/2 \rceil$ appearing in the 
conjecture is the number of summands in the isobaric decomposition.

Although we have not checked this, our conjecture
should follow from the more general 
conjectures of Blasius \cite[Conjecture L.9.8]{blasius2} and 
Panchiskin \cite[Conjecture 2.3]{panchiskin} on the behavior of periods 
of motives twisted by Artin motives.

It appears that the authors may be able to prove a weaker version 
of Conjecture~\ref{con:twisted}, and
so really prove a relation amongst appropriate periods, using Theorem~\ref{thm:mahnkopf-2}
of Mahnkopf; at least in the case when ${\rm Sym}^n(\pi(\varphi))$ is known to exist as a
cuspidal automorphic representation.

\bigskip

\noindent
A.~Raghuram, Department of Mathematics,
Oklahoma State University, 
401 Mathematical Sciences, 
Stillwater, OK 74078, USA. 
E-mail : {\tt araghur@math.okstate.edu}

\smallskip

\noindent
Freydoon~Shahidi, Department of Mathematics,
Purdue University, 150 N. University Street,
West Lafayette, IN 47907, USA.
E-mail : {\tt shahidi@math.purdue.edu}

\end{document}